\documentclass[reqno,12pt]{amsart}

\title[One Dimensional DLA --- Not for circulation]
      {One-dimensional long-range diffusion-limited aggregation III - \\ The limit aggregate}

\author{Gideon Amir}

\date{October 31, 2009}

\usepackage{fullpage}
\usepackage{pst-all,url}
\usepackage{varioref}
\usepackage{amsmath,amsthm,amssymb,amsfonts}
\usepackage{graphicx}
\usepackage{color}
\usepackage{hyperref}

\newcommand{\thmref}[1]{Theorem~\ref{#1}}

\newcommand{\lemref}[1]{Lemma~\ref{#1}}

\newcommand{\propref}[1]{Proposition~\ref{#1}}

\newtheorem{thm}{Theorem}
\newtheorem*{thm*}{Theorem}
\newtheorem{lemma}[thm]{Lemma}
\newtheorem{coro}[thm]{Corollary}
\newtheorem{prop}[thm]{Proposition}
\newtheorem{claim}[thm]{Claim}
\theoremstyle{definition}
\newtheorem{defn}[thm]{Definition}
\newtheorem*{defn*}{Definition}
\theoremstyle{remark}
\newtheorem*{rem*}{Remark}
\newtheorem*{question*}{Question}
\newtheorem{question}{Question}

\newcommand{\ep}{\varepsilon}
\newcommand{\eps}{\varepsilon}
\renewcommand{\P}{\mathbb{P}}

\newcommand{\E}{\mathbb{E}}
\newcommand{\R}{\mathbb{R}}
\newcommand{\Z}{\mathbb{Z}}

\newcommand{\Cc}{\Upsilon}

\renewcommand{\t}{\mathfrak{T}}
\newcommand{\F}{\mathcal{F}}

\DeclareMathOperator{\diam}{diam}
\newcommand{\dist}{d}
\renewcommand{\H}{\mathcal{H}}

\begin{document}
\maketitle

\begin{abstract}

In this paper we study the structure of the limit aggregate
$A_\infty = \bigcup_{n\geq 0} A_n$ of the one-dimensional long range
diffusion limited aggregation process defined in \cite{AABK}. We
show (under some regularity conditions) that for walks with finite
third moment $A_\infty$ has renewal structure and positive density,
while for walks with finite variance the renewal structure no longer
exists and $A_\infty$ has $0$ density. We define a tree structure on
the aggregates and show some results on the degrees and number of
ends of these random trees. We introduce a new "harmonic competition" model where different colours compete for harmonic measure, and show how
the tree structure is related to coexistence in this model.
\end{abstract}

\section{Introduction}

In \cite{AABK} a new $1$-dimensional model of diffusion limited
aggregation(DLA), that tries to capture the fractal nature of the
celebrated DLA model of T.\ Witten and L.\ Sander \cite{WS81},  is
defined and studied. The model, defined rigorously in section
\ref{s:DLA_paths} can be described as follows: Start with an
aggregate containing a single particle at $0$, at each stage, let a
new particle perform a random walk with long jumps starting "from
infinity" until it attempts to jump onto the existing aggregate, at
which stage the jump is not performed and the particle is glued
(added to the aggregate) in its current position. Thus the process
generates a sequence  $\{0\}= A_0 \subset A_1 \subset \ldots$ of
disconnected sets in $\Z$, dubbed the aggregates, with the $n$'th
aggregate $A_n$ having $n+1$ points. In \cite{AABK} and \cite{AAK},
we study the relation between the diameter of the aggregates,
$D_n=\diam(A_n)$ and the step-distribution of the underlying random
walk $R$. More precisely, denoting by $\alpha(R):=\sup\{a\ge 0 : \E
|R_1-R_0|^{a} < \infty\}$ - the highest moment of the walk, it is
shown that under some regularity conditions the diameters exhibit
several phase transitions as the highest moment of the walk varies.
A minimal version of these results is given in the following
theorem:

\begin{thm}[\cite{AABK} Theorem $1$]\label{T:all}
  Let $R$ be a symmetric random walk with step distribution satisfying
  $\P(|R_1-R_0|=k) = (c+o(1))k^{-1-\alpha}$. Let $D_n$ be the diameter of
  the $n$ particle aggregate. Then
  \begin{itemize}
  \item If $\alpha>3$, then $n-1 \le D_n \le Cn+o(n)$ a.s., where $C$
    is a constant depending only on the random walk.
  \item If $2<\alpha\le 3$, then $D_n=n^{\beta+o(1)}$ a.s., where $\beta =
    \frac2{\alpha-1}$.
  \item If $1<\alpha<2$ then $D_n=n^{2+o(1)}$ a.s.
  \item If $\frac{1}{3}<\alpha<1$ then
    \[
    n^{\beta+o(1)} \le D_n \le n^{\beta'+o(1)}
    \]
    a.s., where $\beta=\max(2,\alpha^{-1})$ and $\beta' =
    \frac2{\alpha(2-\alpha)}.$
  \item If $0<\alpha<\frac{1}{3}$ then $D_n=n^{\beta+o(1)}$ a.s.,
    where $\beta=1/\alpha$.
  \end{itemize}
\end{thm}


\begin{figure}
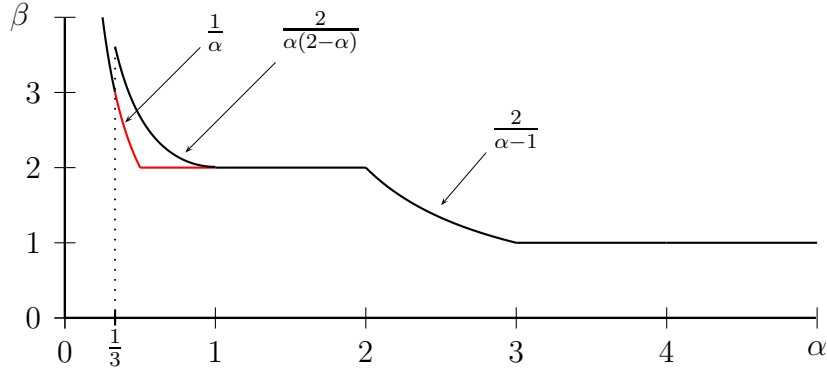

  \begin{center}
    \psset{xunit=20mm,yunit=10mm}
    \pspicture(-1,-.3)(6,4)
    \psaxes(0,0)(4.99,3.99)
    \rput(-.3,4){$\beta$} \rput(5,-.4){$\alpha$}
    \psline(4,1)(5,1)
    \psline(3,1)(4,1)
    \psplot{2}{3}{2 x 1 sub div}
    \psline[linecolor=red](.5,2)(1,2)
    \psline(1,2)(2,2)
    \psplot{.25}{.3333}{1 x div}
    \psplot[linecolor=red]{.3333}{.5}{1 x div}
    \psplot{.3333}{1}{2 x div 2 x sub div .01 add}
    \psline[linestyle=dotted](.3333, 3.6)(.3333,.1)
    \psline(.3333,-.1)(.3333,.1)
    \rput(.3333,-.4){$\tfrac{1}{3}$}
    \psset{linewidth=0.01}
    \rput(1.7,3.8){$\frac2{\alpha(2-\alpha)}$} \psline{->}(1.4,3.4)(.8,2.2)
    \rput(1,3.8){$\frac1{\alpha}$} \psline{->}(0.9,3.6)(.4,2.6)
    \rput(3,2.5){$\frac2{\alpha-1}$} \psline{->}(2.8,2.2)(2.5,1.5)
    \endpspicture
  \end{center}
  \caption{From \cite{AABK}: If the random walk $R$ has $\alpha$ finite moments, then the
    diameter of the resulting $n$-particle aggregate grows as $n^\beta$.
    For $\frac{1}{3}<\alpha<1$ the lower and upper bounds for $\beta$
    differ, and the lower bound is conjectured to be correct.}
  \label{fig:graph}
\end{figure}

For explanations concerning the various phase transitions in the
theorem, as well as a more thorough  introduction to diffusion
limited aggregation processes, the reader is referred to the
introduction of \cite{AABK}.

 In this paper, we study the limit aggregate of the $1$
dimensional diffusion limited aggregation process, defined simply as
$A_\infty=\bigcup_{n\geq 0} A_n$ - the set of all points eventually
added to the aggregates. The natural expectation is that the density
of $A_\infty$ reflects the growth rate of $D_n$, at least to order
of magnitude, that is if $D_n = n^{\beta+o(1)}$ then
\begin{equation}
\big| A_\infty\cap[-n,n] \big| = n^{1/\beta+o(1)}.
\label{eq:Ainfty}
\end{equation}
In this paper, we show that this is indeed the case when $\alpha >
2$, and in doing so provide some further detail into the structure
of $A_\infty$. Our main results are as follows:

\begin{thm}\label{T:density_alpha_geq3}
  Assume $\P(\xi>t) \leq Ct^{-\alpha}$ for any $t$ and some $\alpha>3$.
  There exists some $B>0$ such that a.s.\ $A_\infty$ has density $B$.
  Further, $B$ is the limit density of $A_n$:
  \[
  B = \lim_{m_1,m_2\to\infty} \frac{|A_{\infty}\cap[-m_1,m_2]|}{m_1+m_2}
    = \lim_{n\to\infty} \frac{n}{D_n}.
  \]
\end{thm}

\begin{thm}\label{t:den_0}
Assume there exist $2<\alpha<3$ and constants $c_1,c_2>0$ so that $\xi$ satisfies $c_1n^{-\alpha} \leq \P(\xi\geq n) \leq c_2n^{-\alpha}$  for all $n$
then a.s.
\[|A_\infty \cap [-n,n]| = n^{\frac{\alpha-1}{2}+ o(1)}\]
In particular, $A_\infty$ has $0$ density in the sense that $\lim
\frac{|A_\infty \cap [-n,n]|}{n}=0$.
\end{thm}

To show \thmref{T:density_alpha_geq3}, we first derive upper bounds
on the probability of a random walk passing through a set with $n$
points without hitting it. These bounds, which hold uniformly in the
structure of the set, are then used to show that the process has
renewal times - times at which the subsequent growth of the
aggregate is independent of the structure of the aggregate until
that time. We show that the set of renewal times dominates a renewal
process with positive density, and deduce
\thmref{T:density_alpha_geq3} as a consequence.

 When $2<\alpha<3$, the strategy of "jumping over" a set to avoid it becomes possible.
 Using a lower bound on this probability we show the obstacles created by the growing aggregates
 are not enough to stop new particles from occasionally coming
 through, and conclude that the renewal structure no longer
 exists.
However, by combining properties of the random walk with simple
geometric properties of the aggregates that follow from the diameter
growth rates, we are able to show that it is still hard for
particles to penetrate deep into the aggregate, and derive the upper bound of
\thmref{t:den_0}. The lower bounds follows from directly from the $2<\alpha\le 3$ clause of Theorem \ref{T:all}


 The case $\alpha<2$ has rather different
difficulties and at present we are not ready to speculate on the
validity of \eqref{eq:Ainfty}. However, one must make some
precautions as in Chapter $7$ of \cite{AABK} an example is given of
a walk with $\alpha=0$, ``the $\Z^{3}$ restricted walk'' for which,
despite the fact that $D_n$ grows faster then any polynomial,
$A_\infty=\Z$. We do not know if such examples exist for
$0<\alpha<2$, as the construction used is somewhat special.

Last, we introduce some additional tree structure onto the
aggregates, creating increasing families of random trees which we
call the \textbf{aggregation trees}.

To get the aggregation tree $\t_n$ from the aggregate $A_n$ and the
paths of its particles, we draw an edge between the position at
which each particle was stopped when coming from infinity, and the
position onto which it attempted to jump. The limit aggregation tree
$\t_\infty$ is defined simply as the union of the finite stage
trees. Thus $\t_\infty$ combines spatial information (distance)
together with the graph structure. Two basic questions on the tree
structure are the degree distribution of its vertices and the number
of ends in the tree. In section \ref{s:tree} we give a formal
definition of these trees, relate them to a competition model where
colours compete for harmonic measure and give some partial answers
to the above questions.

\subsection*{Acknowledgements}
The Author wishes to thank Omer Angel, Itai Benjamini and Gady Kozma
for introducing him to diffusion limited aggregation and its
$1$-dimensional variant and for helpful discussions, and to thank
Omer Angel for useful comments on an earlier version of the paper.
This research was supported by the Israel Science Foundation (grant No. 1471) and by a Grant
from the GIF, the German-Israeli Foundation for Scientific Research and
Development.

\section{Preliminaries and notation}
 We will denote a single step of our random walk by $\xi$,
and the random walk itself by $R=(R_0,R_1,\dotsc)$. We will assume
through out the paper that our random walk is aperiodic and
symmetric. We denote by $\P_x$ the probability measure of the random
walk started at $x$. We denote by $p_{x,y}$ the probability of the
random walk to move from $x$ to $y$ in one step (so
$\P(\xi=x)=p_{0,x}$). For a given set $A$, define
\[
p(x,A)=\sum_{a\in A}p_{x,a}.
\]
We denote by $T_A$ be the hitting time of $A$, defined as
\[
T_A = \min\{n>0 \text{ s.t. } R_n\in A\}.
\]
Note that $T_A>0$ even if the random walks starts in $A$. For a set
$A=\{x\}$ with a single member we also write $T_x$.

We define the hitting measure by
\[
H_A(x,a)=\begin{cases}
\P_x(R_{T_A}=a) & x\not\in A\\
\delta_{x,a} & x\in A
\end{cases}
\qquad H_A(\pm\infty,a)=\lim_{x\to\pm\infty}H_A(x,a)
\]
by \cite[T30.1]{S76} the limit on the right-hand side exists for any
aperiodic random walk. $H_A(\pm\infty,\cdot)$ is called the harmonic
measure on $A$ from $\pm\infty$. We will set $H_A(a) = \frac12
H_A(+\infty,a) +\frac12 H_A(-\infty,a)$ and call it the
\textbf{harmonic measure} of $a$ with respect to $A$.

For a subset $A\subset \Z$ we will denote by $\diam A$ the diameter
of $A$, namely $\max A - \min A$, by $\overline{A}=[\min A,\max A]$
the minimal interval containing $A$, and by $|A|$ the number of
points in $A$. For $x\in\Z$ we will denote by $\dist(x,A)$ the
point-to-set distance, namely $\min_{y\in A}|x-y|$. For convenience
we denote by $\Z_+$ the positive integers including $0$ and by
$\Z_-$ the strictly negative integers.

By $C$ and $c$ we denote constants which depend only on the walk $R$
but not on the other parameters involved. The constants hidden in
the $o(\cdot)$ notation may also be random. Generally $C$ and $c$
might take different values at different places, even within the
same formula. $C$ will usually pertain to constants which are ``big
enough'' and $c$ to constants which are ``small enough''.

$X\lesssim Y$ denotes that $X<C Y$. By $X\approx Y$ we mean
$cX<Y<CX$ (that is, $X\lesssim Y \lesssim X$).

In this paper we consider only random walks with finite variance.
The following lemma (Lemma $4.2$ of \cite{AABK}) captures some
properties of such walks that will be useful for our analysis.
\begin{lemma} \label{L:hit_x}
  Let $R$ be a random walk on $\Z$ with steps of mean 0 and variation
  $\sigma^2<\infty$. Then there are $c,C>0$ such that for any $A\subset\Z$,
  $A\neq\emptyset$,
  \begin{enumerate}
  \item\label{enu:limy>c} If $x>\max A$, then $\lim_{y\to\infty}
  \P_y(T_x<T_A) > c$.
  \item\label{enu:escp dist} If $\dist(x,A)$ is large enough then $c <
  \dist(x,A) \P_x(T_A<T_x) < C$.
  \end{enumerate}
\end{lemma}

\subsection{DLA as a measure on infinite paths}\label{s:DLA_paths}

The purpose of this subsection is to define the DLA generated by a
set of random walkers starting "at infinity", in a way that will
retain information on the paths of the particles that were used to
generate the aggregate. This will allow us to study properties of
these paths and relate them to the structure of the aggregates and
the limit aggregate, and in particular allow us to define a renewal
structure on the aggregation process.

We define the measure $\P_{+\infty}$, depending implicitly on $A$,
as follows. This measure is supported on paths $\{\gamma_i\}_{i\leq
0}$, i.e.\ paths with no beginning but a last step. It is defined as
the limit as $y\to\infty$ of the law of $\{R_{T_A+i}\}_{i\leq 0}$.
Informally, $\P_{+\infty}$ is interpreted as the random walk started
at $+\infty$, and stopped when it hits $A$. Clearly it is supported
on paths in $\Z\setminus A$, except for $R_0\in A$. The measure
$\P_{-\infty}$ is defined similarly using $y\to-\infty$. We define
the measure $\P_\infty = \frac12(\P_{+\infty} + \P_{-\infty})$.

It was proved in Lemma $2.1$ of \cite{AABK} that for recurrent
random walks $\P_{+\infty},\P_{-\infty}$ are probability measures
and that for any $x_0 \in A$ and $x_{-1},...x_{-n} \in \Z \setminus
A$
  \begin{equation}
    \P_{\pm\infty}(R_i=x_i \text{ for } -n\le i\le 0) =
    \frac{\P_{\pm\infty} (T_{x_{-n}} < T_A)} {\P_{x_{-n}}(T_A <
      T_{x_{-n}})} \prod_{i=-n}^{-1} p_{x_i,x_{i+1}}
    \label{eq:Ppmrecurr}
    \end{equation}
Let us spend a moment explaining formula \eqref{eq:Ppmrecurr}, as
this type of analysis will return later on in the paper. For
clarity, write $z=x_{-n}$. Now, in order for the event on the
  right-hand side to happen, the walk must hit $z$ before hitting $A$, which happens
  with probability $\P_{\pm\infty}(T_z < T_A)$. By the strong Markov property at
  $T_z$, with probability $\P_z(T_A < T_z)$ the walk will not hit $A$
  before its next return to $z$. Thus the expected number of visits to $z$
  before $T_A$ is
  \[
  \frac{\P_{\pm\infty}(T_z < T_A)} {\P_z(T_A < T_z)}.
  \]
  At each of these visits there is probability $\prod_{i=-n}^{-1}
  p_{x_i,x_{i+1}}$ of making the prescribed sequence of jumps ending at
  $x_0\in A$. Since the walk is stopped once such a sequence of jumps is
  made, the events of making these jumps after the $i$'th visit to $z$ are
  disjoint (for different $i$'s). Hence summing over these events
  gives \eqref{eq:Ppmrecurr}.

To define the aggregates of the $1$-DLA process, it is enough to use
the projection of $\P_\infty$ into the last two steps $R_0,R_{-1}$
of the random walk, as was done in \cite{AABK}. However,
 some of the events we would like to consider, such as
renewal times, will depend on the set of paths used to build the
aggregates and not only on the aggregates themselves. We therefore
define the $1$-DLA process with paths:

\begin{defn}\label{def:DLA_with_paths}
  Let $R$ be a recurrent random walk on $\Z$. The \textbf{DLA process with paths with respect to
    $R$} is a sequence of random tuples $\{A_n,\Pi_1,\ldots
    \Pi_n\}_{n\geq 0}$ where $A_0=\{0\}$, $\Pi_i$ is chosen according
    to $\P_\infty$ with respect to the set $A_{i-1}$, and $A_i =
    A_{i-1} \cup \{\Pi_i(-1)\}$.
The sets $A_0=\{0\} \subset A_1 \subset
  \cdots$ are called the \textbf{aggregates} of the process while
  $\Pi_i$, called the \textbf{path} of the i'th particle, is a
  backward infinite path on $\Z\setminus
  A_{i-1}$ ending at $\Pi_i(0)\in A_{i-1}$. We set $a_0=0$ and
  $a_i=\Pi_i(-1)$, thus $A_n=\{a_0,a_1,\ldots, a_n\}$. We call $a_i$ the $i$'th particle in the
  aggregate. The limit aggregate of the process, $A_\infty$ is defined as the union of all the finite-time aggregates
  $A_\infty=\bigcup_{n\geq 0}A_n = \bigcup_{n\geq 0}\{a_n\}$.\\
  We define  $\F_n$ to be the minimal $\sigma$-field
generated by $\Pi_1,\ldots \Pi_n$. This includes all information
about the aggregates $A_0,\ldots,A_n$.
\end{defn}

It is immediate from the definition that the projection of this
  process onto the sequence $\{A_n\}_n\geq 0$ gives back the usual
  $1$-DLA process defined in \cite{AABK}.

We say that the $i$'th particle \textbf{started from $+\infty$} if
$\lim_{n\to -\infty} \Pi_i(n) = +\infty$ , and that the $i$'th
particle \textbf{started from $-\infty$} if $\lim_{n\to -\infty}
\Pi_i(n) = -\infty$. \\ Since $\P_{\pm\infty} (\lim_{n\to
-\infty}R_n = \pm \infty) = 1$ with probability $1$ every particle
has a well defined starting position.

 Given the paths constructing the $1$-DLA processes, we can now define renewal times:
\begin{defn}
$n$ is called a \textbf{weak right renewal time} for the $1-DLA$
process if \\ $a_n > \max A_{n-1}$, and from time $n$ and on all
particles that start from $+\infty$ are
added to the aggregate to the right of $a_n$. \\
 $n$ is called a \textbf{strong
right renewal time} if in addition the paths of all particles that
start from $+\infty$ after time $n$ do not hit the half line to the
left of $a_n$, nor are any such particles glued to a point to the
left of $a_n$.(i.e. $a_n > \max A_{n-1}$ and $\Pi_i(j)\geq a_n$ for
all times $i\geq n$ at which the i'th particle starts from infinity
and
and for all $j\leq 0$).\\
A symmetric definition holds for left renewal times and particles
coming from $-\infty$.
\end{defn}

\section{The structure of the limit aggregate }

\subsection{Walks with $\alpha > 3$}

Our starting point for analyzing the limit aggregate for the case
$\alpha>3$ is the following theorem from \cite{AABK} bounding
diameter of the aggregates $\{A_n\}$:
\begin{thm}\label{T:thirdmoment} [\cite{AABK} theorem $4.1$]
  If $\E|\xi|^3<\infty$ and $\E\xi=0$ then there is some $C$ so that $\limsup
  \frac{D_n}{n} < C$ a.s.
\end{thm}
As noted in the introduction, this suggests that the limit aggregate
$A_\infty$ might have positive density. Theorem
\ref{T:density_alpha_geq3} state that this is indeed the case. To
prove \thmref{T:density_alpha_geq3} we study the strong renewal
times of the $1-DLA$ process.

\begin{prop}\label{P:renewal_alpha_geq3}
  If $\P(\xi>t)\leq Ct^{-\alpha}$ for some $\alpha>3$ and all $t>0$, then
  there exists a constant $c>0$ depending only on $\xi$, such that for any
  $n$
  \[
  P\left(n \text{ is a strong  renewal time} \ | \ \F_{n-1} \right) > c.
  \]
  Further, the set of renewal times dominates a renewal process with
  positive density $c$.
\end{prop}
 The lower bound given in the proposition is uniform in $\F_{n-1}$
 meaning that it holds for almost all possible paths of the first $n-1$ particles, however
 one should note that the sigma algebra $\F_{n-1}$ does not contain the
information what are the renewal times before particle $n$, as to
know this one must have information on future paths as well.

 To prove the proposition, we will first prove
a general lemma on 1-dimensional random walks with finite variance.
We will need the following definition:

\begin{defn}
  Let $R$ be a random walk on $Z$ with step distribution $\xi$. We define the
  (left-oriented) \textbf{Ladder} process of $R$, denoted $L_R$, to be the
  sequence of distinct values attained by the infimum process $\{inf_{l\leq
    n}R(l)\}_n$. We denote the step distribution of $L_R$ by $L_\xi$.
  Thus $L_R$ has strictly negative steps with distribution equal to the
  hitting measure from 0 of $\Z^-$.
\end{defn}

\begin{lemma}\label{L:overshoot}
  Assume $\E(|\xi|^2)<\infty$. Start the walk $R$ at some $y\geq 0$, and let
  $\tau=T_{Z_-}$. Let $z:=-R(\tau)$ be the overshoot.
  \begin{enumerate}
  \item \label{l:hit_first_point} There exists a constant $c>0$ such that
    $\P_y(z=1)>c$, uniformly in $y\geq 0$ ($z=1$ means the walk hits the half line
    at its rightmost point.) In particular for $y=0$ we find $\P(L_\xi=-1)>c$.
  \item \label{l:Ladder} There exists a constant $C>0$ s.t. for any $k>0\ $
    $\P(L_\xi=-k)\leq C\P(\xi\geq k)$.
  \item \label{l:overshoot} Assume in addition $\P(\xi>t) \leq
    Ct^{-\alpha}$ for some $\alpha>2$. Then there is some $C>0$ s.t.\
    $\P_y(z=k)\leq Ck^{1-\alpha}$ for any $k>0$, uniformly in $y\geq 0$.
  \end{enumerate}
\end{lemma}

Note that by translation invariance similar estimates hold for the
hitting measure of $(-\infty,x)$ from $y$ for any $x$ and any $y\geq
x$.

\begin{proof}
  \begin{enumerate}
  \item By \lemref{L:hit_x}, we know that $\P_y(T_0<T_{Z_-})
    \xrightarrow[y\to \infty]{} c_0 > 0$, so the requisite bound holds for
    all large enough $y>y_0$. However, the walk can make positive steps as
    well, so for any $y>0$ there is a positive probability of exceeding
    $y_0$ before hitting the negatives.

  \item Starting the random walk at 0, $L_\xi=-z$ is the hitting point
    of $\Z_-$. Fix $k>0$, then $\P(z=k)$ is the sum of probabilities
    of all paths from 0 that hit $\Z_-$ at $-k$ and terminate there.
    Partitioning the paths according to the value of $R(\tau-1)$, we get:
    \[
    \P(z=-k) = \sum_{l\geq 0}
               \frac{\P_0(T_l < \tau) \P(\xi=l+k)}  {\P_l(\tau < T_l)}.
    \]
    By \lemref{L:hit_x}, $\P_0(T_l<\tau) \leq \frac{C}{l}$, and
    $\P_l(\tau < T_l) \geq \frac{c}{l}$. Together this gives
    \[
    \P(L_\xi = -z = -k) \leq C\P(\xi\geq k).
    \]

  \item We partition the paths of the ladder walk from $y$ to $-k$ according to the
  place from which the ladder walk made the jump to $-k$.
  Since $L_R$ is strictly decreasing, this gives
    \begin{align*}
    \P_y(z=k)
    &= \sum_{i=0}^y \P(L_R \text{ visits } i) \P(L_\xi=-(i+k)) \\
    &\leq \sum_{i=0}^y \P(L_\xi=-(i+k)) \leq \P(L_\xi\leq -k) < Ct^{1-\alpha},
    \end{align*}
    where the last inequality follows from the bound of (\ref{l:Ladder}) on
    $L_\xi$.
  \end{enumerate}
\end{proof}
 Assume, without loss of generality, that the $n'th$ particles starts from
$+\infty$. Let $x_n$ be the minimal point in the path of the $n$-th
particle and denote $J_n = |A_{n-1} \cap (x_n,\infty)|$, i.e.\ the
number of distinct points in $A_{n-1}$ which the $n'th$ particle has
passed before being added to the aggregate. Thus $J_n$ measures the
amount by which the $n'th$ particle penetrates into the aggregate.

The following lemma is the key ingredient of the proof of
\propref{P:renewal_alpha_geq3}:

\begin{lemma}\label{L:missAalpha3}
  Assume $\xi$ satisfies $\P(\xi>t)\lesssim t^{-\alpha}$ for some $c>0$,
  $\alpha>2$ and all $t>0$. Then there exists a constant $C>0$ s.t.
  \[
  \P(J_n > l \ | \ \F_{n-1}) \leq Cl^{2-\alpha} \log^{\alpha-1}l
  \]
  uniformly in $\F_{n-1}$. In particular the $J_n$'s are stochastically
  dominated by i.i.d.\ random variables with the above tail.
\end{lemma}

\begin{proof}
  The Lemma is easy consequence of the following statement: There
  exist positive constants $c,C$ such that for any $m,k>0$ and for any set $A\subset \Z^+$
with $|A|\geq km$ and any $y>\max A$,
  \begin{equation}\label{e:penetrate_3}
    \P_y(T_{-} < T_A) < e^{-cm} + Cmk^{2-\alpha}.
  \end{equation}
  Indeed, we may shift $A_{n-1}$ so that it has exactly $l$ non-negative
  elements, and apply \eqref{e:penetrate_3} with $A=A_{n-1}\cap\Z^+$, with
  $k=\frac{l}{C log l}$ and $m=C\log l$ for a large enough constant $C>0$.

  To prove \eqref{e:penetrate_3}, we use the bound
  \[
  \P_y(T_{-} < T_A) < \P_y(L_R \text{ avoids } A),
  \]
  since if the ladder process visits $A$, then the corresponding time for
  the walk $R$ is before $T_-$. Henceforth we will only consider the ladder
  process and its steps and time.

  Let $M$ (for many) be the event that there are more then $m$ steps in
  which $L$ jumps over points in $A$ without landing in $A$. Let $B$ (for
  big) be the event that $L$ jumps over at least $k$ points of $A$ in a
  single step. Since $|A|\geq km$, it is clear that in order to miss $A$,
  one of $M$ or $B$ must happen. And therefore
  \[
  \P_y(T_- < T_A) \leq \P_y(M) + P_y(B\cap M^c).
  \]
  Let $\{\rho_i\}_{i=1}^{I}$ be the sequence of times at which the
  ladder walk $L_R$ passes over or hits points in $A$. Formally,
  define $\rho_0=0$ and define inductively
  \[
  \rho_{i+1} = \inf\{t : A\cap [L(t),L(\rho_i))\neq\emptyset\}.
  \]
  This defines a finite sequence since $A$ is finite. Let $\rho_I$ be the
  time the process passes the last point of $A$. Let $\H_k$ be the natural
  filtration of $L_R$. By \lemref{L:overshoot}(\ref{l:hit_first_point}), on
  the event $i<I$ we have $\P(L_R(\rho_{i+1}) \in A | \H_{\rho_i}) > c$.
  Therefore the probability of avoiding $A$ on at least $m$ jumps is
  \begin{equation}\label{eq:PM_small}
  \P(M) \leq \prod_{i=0}^{m-1} 1 - \P\big(L_R(\rho_{i+1}) \in A \ \big | \ \H_{\rho_i}\big)
        \leq e^{-cm}
  \end{equation}

  Let $B_i$ be the event that $|A \cap (L_R(\rho_i),L_R(\rho_{i+1}))| > k$.
  Then for $B_i$ to occur, $L_R$ must make a large step. By
  \lemref{L:overshoot}(\ref{l:overshoot}), $\P(B_i |\H_{\rho_i}
  ) <
  Ck^{2-\alpha}$. we deduce
  \[
  \P_y\big (B_i \ \big | \  \bigcap_{j=0}^{i-1}B_j^c\big) < Ck^{2-\alpha}.
  \]
  The event $B\cap M^c$ implies that at least one of the events
  $B_0,\dotsc,B_{I-1}$ occurs, and that $I<m$, therefore
  \[
  \P_y(B\cap M^c) \leq \sum_{i=0}^I \P_y(B_i | \bigcap_{j=0}^{i-1}B_j^c)
  \leq m Ck^{2-\alpha}.
  \]
  Together with \eqref{eq:PM_small} this yields \eqref{e:penetrate_3}.
\end{proof}

\begin{proof}[Proof of \propref{P:renewal_alpha_geq3}]
  By symmetry it suffices to prove the proposition for right renewal times. Let
  $\{n_i\}$ be the sequence of times at which particles start from
  $+\infty$, and let $\hat J_i=J_{n_i}$ be the amounts by which
  these particles penetrate the aggregate.
  By \lemref{L:missAalpha3}, $\hat J_i$ are stochastically dominated by
  i.i.d.\ random variables: $\hat J_i\leq Y_i$ with
  \[
  \P(Y_i>t) < Ct^{2-\alpha+\ep}.
  \]
  Where $\ep>0$ is such that $2-\alpha+\ep<-1$ (or $\ep<\alpha-3$).
  Also, by \lemref{L:overshoot}(\ref{l:hit_first_point}), we can have
  $\P(Y_i=0)>c>0$.

  First, we claim that a sufficient condition for $n_k$ to be a strong
  right renewal time is that $\hat J_i \leq i-k$ for all $i\geq k$. This
  follows by induction on $i$: - if at time $n_i$ there are at least $i-k$
  particles to the right of $a_{n_k}$, and $J_{n_i} < i-k$ then the
  path of $a_{n_i}$ does not pass $a_{n_k}$ and $a_{n_i}$ is added to
  the right of $a_{n_k}$. Note that particles arriving from $-\infty$ do
  not pose a problem here for two reasons. First, these can only increase
  the number of points to the right of $a_{n_k}$ and second, after the
  first left strong renewal time even that will not occur.

  Consider an infinite family of i.i.d.\ variables $\{Y_i\}_{i\in\Z}$ with
  distribution as above, coupled with the variables $\{J_i\}_{i\in \Z_+}$ so that $Y_i\geq J_i$ for $i\geq 0$. Define
  \[
  \Cc = \{n : \forall m\geq n, Y_m\leq m-n\}.
  \]
  Equivalently, $\Cc$ is the complement of the union of (open) intervals
  $\bigcup (m-Y_m,m)$. Clearly $\Cc$ is a translation invariant renewal
  process, and by the above, $\Cc\cap\Z^+$ is a subset of the strong right
  renewal times.

  To bound from below the probability that $n$ is a strong renewal time, we find
  the density of $\Cc$. By translation invariance this is
  \[
  \P(0\in\Cc) = \P\left(\forall m\geq0, Y_m<m\right)
             \geq \prod_{m\geq 0} c \vee (1-(n-m)^{2-\alpha+\ep})
             > 0.
  \]
\end{proof}



We are now ready to prove \thmref{T:density_alpha_geq3}.
\begin{proof}[Proof of \thmref{T:density_alpha_geq3}]
  First, observe that after the first right and left strong renewal times,
  no particle coming from $-\infty$ affects the growth of the right side of
  the aggregate. By symmetry it is therefore enough to consider only particles
  coming from $+\infty$ and show that
  \[
  \lim_{m\to\infty} \frac{|A_\infty \cap [0,m]|}{m}
  = \lim_{n\to\infty} \frac{n}{D_n}
  \]
  exists and is a.s.\ constant.

  Let $\{r_k\}_{k\geq 1}$ be the sequence of all strong right renewal
  times. Denote $w_k = r_{k+1}-r_k$: the number of particles in the
  $k'th$ renewal interval. Since the growth to the right of the aggregate
after time $r_k$ does not depend on the history before time $r_k$,
after the first left renewal time, $w_k$ form an i.i.d.\ sequence.
By \propref{P:renewal_alpha_geq3} and the Renewal Theorem, $\E w_k <
\infty$, and therefore by the law of large numbers $r_k = B_1 k(1+o(1))$
for some constant $B_1$.

  Denote $d_k = \max A_{r_{k+1}} - \max A_{r_k}$: the diameter of the
  $k'th$ renewal interval. $d_k$ also form an i.i.d.\ sequence. As noted in
  the proof of \thmref{T:thirdmoment}, $d_k$ is stochastically dominated by
  $\sum_{i=1}^{w_k} Y_i$ where $Y_i$ are i.i.d.\ variables with finite
  expectation, and thus $\E d_k < \infty$. We can apply the Renewal Reward
  Theorem (see e.g.\ \cite{durrett2010probability}) to conclude that $\sum^k d_i = B_2
  k(1+o(1))$ for some constant $B_2\geq 1$.

  The result follows (with $B=B_1/B_2$) for the subsequence where $n$ is a
  renewal time and $m = \max A_n$ at the renewal times. Existence of the
  limit over all $m$ and $n$ follows by sandwiching $m$ between two renewal
  points (or $n$ between two renewal times.)
\end{proof}

Note that while the above theorem proved that the aggregate has
positive density, it also easy to see that under the conditions of
the theorem, if $R$ is not a simple random walk, the density of the
limit aggregate is not $1$, as holes may happen in any renewal
interval.

\subsection{Walks with $2<\alpha<3$}
 When $2<\alpha(R)<3$, the structure of the limit
aggregate is quite different from the case $\alpha>3$, as will be
seen in theorem \ref{t:den_0} and in claim
\ref{c:finite_renewal_23}.

Our starting point will be once again the bounds on the diameters
$D_n$ given in \cite{AABK}. The following theorem summarizes the
lower bound (\cite{AABK} theorem $5.1$) and the upper bound
(\cite{AABK} theorem $5.3$)
 for our special case:
\begin{thm} \label{T:23_diam}
  Fix $\alpha\in(2,3]$ and let $\beta = \frac{2}{\alpha-1}$. If the random
  walk is such that $\P(\xi>t) \approx t^{-\alpha}$ and $\E\xi=0$,
  then a.s.\ $\max A_n = n^{\beta+o(1)}$ and $-\min A_n = n^{\beta+o(1)}$.
\end{thm}
Note that while the lower bound (\cite{AABK} theorem $5.1$) was
stated for $\diam(A_n)$, the proof dealt separately with $\max A_n$
and $\min A_n$.

We first show that the renewal structure that existed for $\alpha>3$
no longer holds:
\begin{claim}\label{c:finite_renewal_23}
If $\P(\xi>t) \approx t^{-\alpha}$ for some $2 < \alpha < 3$ then
there are only finitely many weak renewal times in the $1-DLA$
process.
\end{claim}

The main tool for proving the claim, is a lower bound on the
probability of hitting a set $A$ while avoiding a set $B$ to its
right, given in the next Lemma:

\begin{lemma}\label{l:jump_over}
Assume  $P(\xi>t) \geq ct^{-\alpha}$ $\forall t\in \R_+$ for some
constant $c>0$, and $E(|\xi|^2) < \infty$.  There is a constant
$c>0$,
 such that for any two finite sets of $A,B$
satisfying $\max A < \min B$

$$\P_{\infty}(R(T_{A\cup B}-1) \in \overline{A}) \geq c\cdot (\diam(A) - |A|)\cdot \diam(A\cup B)^{1-\alpha}$$
\end{lemma}
\begin{proof}
By \lemref{L:overshoot}(\ref{l:hit_first_point}), there is some
$c>0$ (independent of $A$,$B$), such that if the random walker is at
$\overline{A}\setminus A$, it will hit $A$ before leaving
$\overline{A}$ with probability $\geq c$. Therefore it is enough to
bound from below the probability that a random walker from infinity
hits $\overline{A}\setminus A$ before hitting $\overline{B}\cup A$.

We bound this by decomposing the set of all paths from infinity that
hit $\overline{A}$ without hitting $\overline{B}$ according to the
place from which the walker jumps to $\overline{A}$, we get (using
the same counting argument as in \eqref{eq:Ppmrecurr})
\begin{eqnarray}
\lefteqn{\P_\infty(T_{\overline{A}\setminus A} < T_{\overline{B}\cup
A}) =  \sum_{z\notin
 \overline{A}\cup \overline{B}}\frac{\P_\infty(T_z <
T_{\overline{A}\cup\overline{B}})p(z,
(\overline{A}\setminus A))}{\P_z(T_{\overline{A}\cup\overline{B}} < T_z)}}  && \nonumber \\
  & \geq  \sum_{z\geq \max{(A\cup B)} + \diam(A\cup
B)}\frac{\P_\infty(T_z < T_{\overline{A}\cup\overline{B}})p(z,
\overline{A}\setminus A)}{\P_z(T_{\overline{A}\cup\overline{B}} < T_z)} & \nonumber \\
 & \geq  \sum_{z\geq \max{(A\cup B)} + \diam(A\cup B)} c
\dist(z,A\cup B)
p(z, \overline{A}\setminus A)  &\label{use_distance}\\
 & =  \sum_{a\in \overline{A}\setminus A} \left(\sum_{z\geq \max{(A\cup B)}
 + \diam(A\cup B)} c \dist(z,A\cup B) p_{z, a} \right) &\\
  & \geq   c\, |\overline{A}\setminus A|\, (\diam(A\cup B) \, \P(\xi\geq 2\cdot \diam(A\cup B))  &\label{need_good_dist}\\
 & \geq  c\cdot (\diam(A) - |A|)\cdot \diam(A\cup B)^{1-\alpha}&
\end{eqnarray}
The inequality in \ref{use_distance} following by using both parts
of \lemref{L:hit_x}.
%
\end{proof}

\begin{proof} [Proof of claim \ref{c:finite_renewal_23}]:
We now use our bounds on the diameter of $A_n$ to prove the claim:

Fix some $\ep>0$.   Let $n_0$ be the (random) minimal $n_0$ such
that for all $n>n_0$
$$n^{\frac{2}{\alpha-1}-\epsilon} \leq \diam(A_n) \leq
n^{\frac{2}{\alpha-1}+\epsilon}.$$  By \thmref{T:23_diam},
$n^{\frac{2}{\alpha-1}-o(1)} \leq \diam(A_n) \leq
n^{\frac{2}{\alpha-1}+o(1)}$ a.s. , so $n_0$ is finite a.s.

Fix $n$. Let $E^w_k$ ($k\geq 0$) be the event that the particle at
time $n+k$ either starts at $-\infty$ or is glued to the right of
$A_{n-1}$, and that $ (n-1)^{\frac{2}{\alpha-1}-\epsilon} \leq
\diam(A_{n-1}) \leq (n-1)^{\frac{2}{\alpha-1}+\epsilon}.$
 If $n$ is a weak
right renewal time and $n>n_0$, then $\bigcap_{k\geq 0}E^w_k$ must
occur. Therefore
$$\P(\text{n is a weak renewal time and } n>n_0) \leq \P(E^w_0)\prod_{k\geq 1}
\P(E^w_k \ | \ \bigcap_{j=1}^{k-1}E^w_j)$$ Let $B_k =
\{a_n,a_{n+1},..,a_{n+k-1}\}$. By applying Lemma \ref{l:jump_over}
with respect to the sets $A=A_{n-1}$ and $B=B_k$, we get that
\begin{eqnarray*}
\lefteqn{P(E^w_k \ | \ \bigcap_{j=1}^{k-1}E^w_j) } &&\\
&\leq & 1 -  c\cdot (\diam(A_{n-1}) -
|n-1|)\cdot \diam(A_{n+k-1})^{1-\alpha} \\
&\leq & 1 -
cn^{\frac{2}{\alpha-1}-\epsilon}(n+k)^{(\frac{2}{\alpha-1}+\epsilon)(1-\alpha)}\\
&\leq&1 - c n^{\frac{2}{\alpha-1}}(n+k)^{-2}(n+k)^{-\epsilon \alpha}
\end{eqnarray*}

And therefore
\begin{eqnarray*}
\lefteqn{P(\text{n is a weak renewal time and }n>n_0) \leq
P(E^w_0)\prod_{k\geq 1} P(E^w_k \ | \ \bigcap_{j=1}^{k-1}E^w_j) }&&\\
& \leq & \prod_{1\leq k \leq n}P(E^w_k \ | \ E^w_{k-1}) \\
& \leq & \prod_{1\leq k \leq n}\left(1 - c
n^{\frac{2}{\alpha-1}}(n+k)^{-2}(n+k)^{-\epsilon \alpha}\right)
\\ &\leq& \left(1 - cn^{\frac{2}{\alpha-1} - 2 - \epsilon\alpha}
\right)^n  \\ & \leq & e^{-n^{\frac{2}{\alpha-1}-1-\epsilon \alpha}}
\end{eqnarray*}

Thus for any $2<\alpha<3$ and $\epsilon$ small enough \[\sum_{n\geq
1}P(\text{n is a weak renewal time and }n>n_0) < \infty\] So by the
Borel-Cantelli lemma there are a.s. only finitely many weak renewal
times bigger than $n_0$. Since $n_0$ is finite a.s., there are a.s.
only finitely many weak right renewal times. The case of left
renewal times follows by symmetry.
\end{proof}

Next we will prove an upper bound on the probability of a random
walk hitting $A_n$  at time $m>n$, without hitting the points glued
to the right or to the left of $A_n$. To get this bound we will need
some geometric properties of $A_n$ (Unlike the lower bound in Lemma
\ref{l:jump_over} which holds for general sets).

The following definition captures the geometric property we will
need, which roughly means that the set does not have big gaps
between its points, where the gaps are measured on the scale of
their position on $\Z$.

\begin{defn}
   A set of positive integers $B=\{b_1<\dots<b_k\}$ is said to be
   $\eps$-dense if $b_{i+1}<b_i^{1+\eps}$ for each $i$. The set is said
   to be dense in an interval $[n,m]$ if $\{n,m\}\cup(B\cap[n,m])$ is
   $\eps$-dense. A set $B\subset\Z^-$ is said to be dense if $-B$ is
   dense.
\end{defn}
We remark that by this definition the empty set, and any singleton,
are also considered $\eps$-dense, as they do not contain any gaps.

\begin{lemma}\label{L:no_penetrate}
  Assume $\P(\xi>t) < ct^{-\alpha}$ for some $2<\alpha<3$, and fix
  $0<\eps<\alpha/2-1$. Let $B_+\subset \Z_+$ and $B_-\subset \Z_-$ be finite
  sets. Let $B=B_-\cup B_+$ and set $m=\min(\max B_+,\min B_-)$,and $M=\max(\max B_+,-\min B_-)$.
  Suppose that $m>2n$, $B_+$ is $\eps$-dense in $[n,\max B_+]$ and that
  $B_-$ is $\eps$-dense in $[\min B_-,-n]$.
  Then for some $C>0$ depending only on $\eps$ and the walk (and not on $B$ or $n$), for
  any $y$ with $|y|>M$
  \[
  \P_y(T_{[-n,n]} < T_B)  <  C n^{1+\eps} m^{1-\alpha}.
  \]
\end{lemma}

\begin{proof}
  Define
  \[
  f(k) = \begin{cases}
    \sup_{x\geq k} \P_x(T_{[-n,n]} < T_B)  & k\geq0, \\
    \sup_{x\leq k} \P_x(T_{[-n,n]} < T_B)  & k<0.
  \end{cases}
  \]
  We will prove by induction on $|k|$ that for a suitable $\gamma$ and any
  $k\in B$ with $|k|>2n$,
  \begin{equation}\label{e:f_bound}
    f(k)\leq \gamma n^{1+\ep} k^{1-\alpha}.
  \end{equation}
  The Lemma follows by considering $k = \max B$ and $k=\min B$.

  By symmetry we may assume $k>2n$. As the base of our induction we will first prove the statement for
$2n < k \leq n^{1+\ep}$. Fix $x>k$ and let $E_k$ be the event that
the random walker
  does not hit $k$ before hitting the half line $(-\infty,\frac{k}{2})$. Our first task is to bound
  $\P(E_k)$. Let $z$ be the first point at which the walk hits
  $(-\infty,k)$. Partitioning according to $z$, we have for any $x\geq k$
  \[
  \P_x(E_k) \leq \sum_{i=\frac{k}{2}}^{k}\P_x(z=i)\P_i(E_k) +
  \P_x\left(z \leq \frac{k}{2}\right).
  \]
  By \lemref{L:overshoot}(\ref{l:overshoot}), $\P_x(z=i) \leq c(k-i)^{1-\alpha}$, and
  $\P_x(z\leq \frac{k}{2}) \leq ck^{2-\alpha}$. By \lemref{L:hit_x},
  $\P_i(E_k) \leq c\frac{k-i}{k}$ for any $\frac{k}{2}\leq i \leq k$.
  Combining these bounds gives
  \begin{equation}\label{e:Ek}
  \P_x(E_k) \leq \sum_{i=k/2}^{k} \frac{c(k-i)^{2-\alpha}}{k} +
  ck^{2-\alpha} \leq ck^{2-\alpha}.
  \end{equation}

For $k\in B$, $k > 2n$ implies $\{T_{[-n,n]} < T_B\} \subset E_k$
and thus for any $k\in B$ with $ 2n < k \leq n^{1+\ep}$
\[f(k) \leq \sup_{x\geq k}\P_x(E_k) \leq c k^{2-\alpha} \leq c
n^{1+\ep} k^{1-\alpha} .\]

Giving us the basis for our induction (for $\gamma > c$).\\

It remains to show using induction on $k$ that \eqref{e:f_bound} holds for $k>n^{1+\ep}$.\\
Let
 $\tau =
  T_{(-\infty,k/2)}$ denote the hitting time of $(-\infty, \frac{k}{2})$ by
  the random walk.  \\
  Denote the hitting point by $y=R_\tau$ and define the
  events
  \begin{align*}
  Q_1 &= \left\{y    < -\frac{k}{4} \right\},  &
  Q_2 &= \left\{|y|\leq \frac{k}{4} \right\},  &
  Q_3 &= \left\{y     > \frac{k}{4} \right\}.
  \end{align*}
  We have
  \begin{equation}\label{e:Qs}
  \P_x(T_{[-n,n]} < T_B) = \sum_{i=1}^{3} \P_x(Q_i, T_{[-n,n]} < T_B).
  \end{equation}

  We bound each of the three summands in terms of the value of $f$ at
  smaller $x$. By \lemref{L:overshoot} we have $\P_x\left(y <
    -\frac{k}{4}\right) < ck^{2-\alpha}$. By the definition of $f$, this
  implies
  \begin{equation}\label{e:Q1}
    \P_x(Q_1, T_{[-n,n]} < T_B) < ck^{2-\alpha} f\left(-\frac{k}{4}\right).
  \end{equation}
  Similarly, for any $i\in\left[-\frac{k}{4},\frac{k}{4}\right]$ we have
  $\P(y=i)< ck^{1-\alpha}$, and so
  \begin{equation}\label{e:Q2}
    \P_x(Q_2, T_{[-n,n]} < T_B) < ck^{1-\alpha} \sum_{i=-k/4}^{k/4} f(i).
  \end{equation}

  To bound the third summand, note that $Q_3\cap\{T_{[-n,n]}<T_B\} \subset E_k$, and so by \eqref{e:Ek} and the Markov
  property at time $\tau$,

  \begin{align}
  \P_x(Q_3,T_{[-n,n]} < T_B) &\leq \P_x(E_k) \max_{\frac{k}{4} < z <
    \frac{k}{2}} \P_z(T_{[-n,n]} < T_B)             \nonumber \\
  &\leq ck^{2-\alpha} f\left(\frac{k}{4}\right).    \label{e:Q3}
  \end{align}

  Combining \eqref{e:Qs}--\eqref{e:Q3} gives
  \begin{align}\label{eq:fk}
    f(k) & < ck^{2-\alpha} f\left(-\frac{k}{4}\right)
           + ck^{2-\alpha} f\left(\frac{k}{4}\right)
           + ck^{1-\alpha} \sum_{i=-k/4}^{k/4} f(i) \\ \nonumber
         & < ck^{1-\alpha} \left( k f\left(-\frac{k}{4}\right)
           + k f\left(\frac{k}{4}\right) + \sum_{i=-k/4}^{k/4} f(i)
         \right).
  \end{align}
We will now want to bound $f(i)$ for $i < k/4$. We will first assume
$i> n^{1+\ep}$.
 At this point we use the fact that $B_\pm$ are $\eps$-dense (and therefore $B\cap
(i^{\frac{1}{1+\ep}},i) \neq \emptyset$.) and that $f$
  is by definition decreasing on $\Z^+$ (and decreasing on $\Z^-$).
  Together with the induction hypothesis these facts imply
  \[
  f(i) \leq \gamma n^{1+\eps} |i|^{\frac{1-\alpha}{1+\eps}} < \gamma
  n^{1+\eps} |i|^{1-\alpha+2\eps}
  \qquad \text{ as long as $|i|>n^{1+\eps}$.}
  \]
  For $i<n^{1+\ep}$ we will use the trivial bound $f(i)\leq 1$.

  Using these bounds in \eqref{eq:fk} gives
  \begin{align*}
  f(k) &<  ck^{1-\alpha} 2n^{1+\eps}
         + ck^{1-\alpha} \gamma n^{1+\eps} \left(2k (k/4)^{1-\alpha+2\eps} +
           2\sum_{i=n^{1+\eps}}^{k/4} i^{1-\alpha+2\eps} \right) \\
    &< ck^{1-\alpha} 2n^{1+\eps} + ck^{1-\alpha} \gamma n^{1+\eps} \left(
       k^{2-\alpha+2\eps} + n^{(1+\eps)(2-\alpha+2\eps)} \right).
  \end{align*}
  Here $c$ is some constant depending only on the random walk. To get the
  claimed bound on $f(k)$ we need this to be less than $\gamma n^{1+\eps}
  k^{1-\alpha}$. This happens iff
  \[
  2c + \gamma c(k^{2-\alpha+2\eps} + n^{(1+\eps)(2-\alpha+2\eps)} )
  < \gamma.
  \]
  We can easily find $\gamma,n_0$ so that this holds for any $k>n>n_0$.
\end{proof}

%
%
We now have all the pieces to prove \thmref{t:den_0}
\begin{proof}[Proof of \thmref{t:den_0}]
For the lower bound on $A_\infty\cap [-n,n]$ it is enough to note that by Theorem \ref{T:all} the first $n^{\frac{\alpha-1}{2}+o(1)}$ points of the aggregate a.s. lie in the interval $[-n,n]$.  We are therefor left with showing the upper bound.

Fix $\ep>0$, which we may assume to be small enough. It is enough to prove that there a.s. exists a constant $C$ so that
$|A_\infty \cap [-n,n]| \leq Cn^{\frac{\alpha-1}{2}+ \ep}$ for all $n$. We will first show that the
aggregates $A_m$ satisfy the geometric property required for
\lemref{L:no_penetrate} for large enough $m$:

\begin{lemma}\label{l:dense}
For any $\ep>0$ there is some $n_0$ s.t. for any $m>n>n_0$
\begin{enumerate}
\item $m^{\frac{2}{\alpha-1}-\ep} \leq  \max A_m \leq
m^{\frac{2}{\alpha-1}+\ep}$
\item $m^{\frac{2}{\alpha-1}-\ep} \leq  -\min A_m \leq
m^{\frac{2}{\alpha-1}+\ep}$
\item $A_m$ is $\ep$-dense in $[\min A_m,\min A_n]$ and in $[\max A_n,\max A_m]$.
\end{enumerate}
\end{lemma}
\begin{proof}
The first two clauses are just a restatement of \ref{T:23_diam}. To
prove the third clause, observe that if there is no point in $A_m$
between $k$ and $k^{1+\ep}$ (or between $-k$ and $-k^{1+\ep}$), then
the diameter of the aggregates must grow by too much in one step,
contradicting the diameter bounds.
\end{proof}
We now return to the proof of theorem \ref{t:den_0}. Fix some
$n>n_0$. For any $m>n$ \[\P(a_{m+1} \in \overline{A_n}) \leq
\sup_{|x|\geq \max_{a\in A_m}|a|}\P_x(T_{\overline{A_n}} < T_{A_m
\setminus \overline{A_n}})\]

 Take $m_0 = 8n^{1+3\ep}$. By the above lemma $\max A_n\leq n^{\frac{2}{\alpha-1} +
\eps}$,$-\min A_n \leq n^{\frac{2}{\alpha-1} + \eps}$ and also $\max
A_m\geq m^{\frac{2}{\alpha-1} - \eps}$,$-\min A_m \geq
m^{\frac{2}{\alpha-1} - \eps}$. In particular for any $m>m_0$, $\max
A_m \geq 2 max A_n$, $-\min A_m \geq -2 \min A_n$, and by the third
clause of Lemma \ref{l:dense}  $A_m$ is $\ep$-dense in $[\min
A_m,\min A_n]$ and in $[\max A_n,\max A_m]$

thus we can apply Lemma \ref{L:no_penetrate} to get

\[\sup_{|x|\geq\max_{a\in A_m}|a|}\P_x(T_{\overline{A_n}} < T_{A_m \setminus \overline{A_n}}) \leq \]  \[ \leq
 C n^{(\frac{2}{\alpha-1}+\ep)(1+\eps)} (m^{\frac{2}{\alpha-1}-\ep})^{1-\alpha}
\leq C n^{\frac{2}{\alpha-1}+ 4\ep}m^{-2+2\ep}\]

Let $W_{m_0,n}$be the number of points added to $[\min A_n,\max
A_n]$after time $m_0$. Then
\[\E(W_{m_0,n}) \leq \sum_{m\geq m_0} C n^{\frac{2}{\alpha-1}+4\ep}m^{-2+2\ep} \leq
Cn^{\frac{2}{\alpha-1}+4\ep} m_0^{2\ep-1} \leq
Cn^{\frac{2}{\alpha-1} -1 + 9 \ep}\]

Since our bounds are uniform in the history for all $m>m_0$, and
they are all indicator variables, the variance of their sum is less
then the expectation, and by Chebyshev's inequality

 $$\P(W_{m_0,n} > 2C n^{\frac{2}{\alpha-1}
-1 +9\ep}) \leq \frac{C}{n^{\frac{2}{\alpha-1} -1 +9\ep}}$$
  and therefore there is some $N_0$ such that this does not happen for any
$n>N_0$. Thus for any $n>N_0$ we add at most $m_0 = 8n^{1+3\ep}$
points until time $m_0$, and $Cn^{\frac{2}{\alpha-1} -1 +9\ep}$
after time $m_0$. Since $\ep$ was arbitrary, we are done.
\end{proof}

\section{The aggregation tree}\label{s:tree}

We start by giving the definition of the aggregation trees promised
in the introduction
\begin{defn}
Given a $1$-DLA process  $\{A_n,\Pi_1,\ldots
    \Pi_n\}_{n\geq 0}$ we define the $n'th$ \textbf{aggregation tree} $\t_n$ as
    the graph whose vertices are the points in $A_n$, and whose
    edge set is  $\{(\Pi_i(0),\Pi_i(-1))\}_{i\leq n}$. The
    limit aggregation tree is defined by $\t_\infty=\bigcup_{n\geq
    0}\t_n$. These trees can be given a directed structure by
    directing each edge from $\Pi_i(0)$ to $\Pi_i(1)$.
\end{defn}

Note that by the definition of the aggregation tree, the probability
of adding an edge from a vertex $n\in \t_k$ at time $k+1$ is
$\P_\infty(R(0)=n) = H_{A_k}(n)$ - the harmonic measure of $n$ with
respect to $A_k$.

We will be interested in two types of properties on $\t_\infty$ -
the degrees of its vertices and the number of ends it possesses.

The number of \textbf{ends} in a tree is the maximal cardinality of
the number of vertex-almost-disjoint infinite simple paths in the
tree. ( Almost disjoint meaning every pair of paths sharing only
finitely many vertices).

 This
can also be thought of in terms of a coexistence of different
species in a competition model: Start by choosing $n$ and colouring
the points of $\t_n$ with different colours  and then grow the
aggregation tree using the DLA dynamics, colouring each new point by
the colour of the point to which it was glued. A colour is said to
survive if its component in $\t_\infty$ is infinite. It is not hard
to see that the maximal colours that can coexist is $\geq n$ if and
only if the number of ends in $\t_\infty$ is $\geq n$.

The renewal structure for $\alpha>3$ easily implies the number of
ends.

\begin{lemma}
If $\P(\xi>t) \lesssim t^{-\alpha}$ for $\alpha>3$ then a.s.
$\t_\infty$ has $2$ ends.
\end{lemma}
\begin{proof}

Let $\tau_i$ be the right strong renewal times. Any path $P$ in
$\t\infty$ which is not bounded above ($|P\cap Z_+|=\infty$ has to
include all $a_{\tau_i}$ for all large enough $i$. Therefore there
cannot be $2$ almost disjoint such paths. The renewal times also
ensure that there are two disjoint paths one going to $+\infty$ and
one to $-\infty$, so $\t_\infty$ has $2$ ends.
\end{proof}

\begin{rem*}
To understand the difficulties in finding the number of ends without
the renewal structure, we take another look at the competition
model. At each stage, the point to which the next particle is glued
is distributed according to the harmonic measure on the aggregate.
Thus the colour of the next particle is distributed proportional to
the harmonic measure of each colour with respect to the aggregate.
When a Red particle is added, the harmonic measure the red part of
the aggregate increases, while the harmonic measure of all other
colours decreases. Thus the colours can be thought of as competing
for harmonic measure. For Red to die out, its harmonic measure must
decrease so that $\sum_{n\ge 0}H_{A_n}(Red)$ is finite, otherwise
red points will occur infinitely often almost surely. The problem is
that a small number of points added to the aggregate can make a big
difference in the harmonic measure on the aggregate. Consider a
random walk $R$ with finite variance. Take the competition process
between Red and Blue, and assume that at time $n$ a new red point is
added to the right of $A_{n-1}$. The harmonic measure of that point
is at least a constant. Therefore even if Red started to die out,
one point is enough to increase its harmonic measure to at least a
constant.
\end{rem*}

In light of the above we ask
\begin{question}
How many ends does $\t_\infty$ have for walks with $2<\alpha<3$
\end{question}

Two other natural questions which hold in many models of random
trees are the following:
\begin{question}
Is the number of ends in $\t_\infty$ a.s. constant (or infinite) for
any random walk $R$?
\end{question}
\begin{question}
Is the number of ends in $\t_\infty$ always in $\{1,2,\infty\}$ ?
\end{question}
The difficulty comes from the high dependency structure between
particles added at different steps to the aggregate.

We now go on to study some properties of the degrees of the vertices
in $\t_\infty$.

The first lemma shows that for walks with finite variance the
degrees are not uniformly bounded.
\begin{lemma}
If $\E(|\xi|^2) < \infty$ then $\sup_{v\in\t_\infty} deg(v) =
\infty$
\end{lemma}
\begin{proof}
Every time a particle $a_k$ is added to the right of the aggregate,
there is a positive probability depending only on $n$ (and not on
$A_k$) that the next $n$ particles will all come from $+\infty$ and
glue to $a_k$ without ever jumping over $a_k$. Since particles glue
on the right i.o., we will almost surely have vertices with degree
$>n$ for any $n$.
\end{proof}

We now show that under general conditions, having a vertex with
infinite degree implies $A_\infty = \Z$:
\begin{prop}\label{p:infinite}
Assume $\E(|\xi|)<\infty$ and for every $k\in \Z$ there exist
constants $c_k,C_k>0$ such that $c_k < \frac{p_{x,k}}{p_{x,0}} <
C_k$ for all $x\neq 0,k$.

Then for any $n_0\in \t_\infty$ $\P(\{deg_{\t_\infty}(n_0) = \infty
\} \cap \{n_1\notin A_\infty \ \text{or}\  \deg_{\t_\infty}(n_1) <
\infty\})= 0$. And in particular if there a.s. exists some $n_0\in
\t_\infty$ with infinite degree then a.s. all vertices in
$\t_\infty$ have infinite degree and $A_\infty = \Z$ a.s. .

\end{prop}
\begin{proof}

We will use Levy's extension to the Borell-Cantelli Lemma (see e.g.
\cite{williams1991probability}, Thm 12.15)
\begin{lemma}
Let $E_n$ be a sequence of events. And let
$\mathcal{G}_n=\sigma\{E_1,\ldots,E_n\}$. Then $$\P( \{\limsup E_i\}
\triangle \{\sum_{n\geq 0} \P(E_i \, | \, \mathcal{G}_{n-1})\}) =
0$$

\end{lemma}

Assume the degree of $n_0$ in $\t_\infty$ is infinite. Let $t_o$ be
the time at which it was added to the aggregate. Let $I_k$ be the indicator of the
event that an edge is connected to $n_0$ at step $k+1$. Then $\P(I_k=1 | \mathcal{G}_{k-1}) = H_{A_k}(n_0)$, and the degree
of $n_0$ in $\t_\infty$ is just $1+\sum_{k\geq t_0} I_k$.

By Levy's extension to the Borell-Cantelli
$\P(\{deg_{\t_\infty}(n_0)=\infty \} \triangle \{\sum_{k\geq t_0}
H_{A_k}(n_0)= \infty\}) = 0$. Fix any $n_1\in \Z$, and let $t_1$ be
a time at which both $n_0$ and $n_1$ are in $A_{t_1}$. Decomposing
paths from $\pm\infty$ to $n_1$ by the position from which they
jumped, we get,by \eqref{eq:Ppmrecurr} that for any set $A$ with
$n_0,n_1\in A$
 \begin{eqnarray*}H_A(n_1) = \P_\infty(R(0) = n_1) = \sum_{x\notin
A}\P_\infty(R(0)=n_1,R(-1)=x) \\ = \sum_{x\notin
A}\frac{\P_\infty(T_x < T_A)}{\P_x(T_A < T_x)}p_{x,n_1} \approx
\sum_{x\notin A}\frac{\P_\infty(T_x < T_A)}{\P_x(T_A <
T_x)}p_{x,n_0} = H_A(n_0)\end{eqnarray*}

Where the constants in the $\approx$ come from the condition on the
walk (and do not depend on $A$). We deduce
$\P(\{deg_{\t_\infty}(n_0)=\infty \} \triangle \{\sum_{k\geq
t_1}H_{A_k}(n_1)= \infty\}) = 0$
  which implies,
using again Levy's extension to the Borell-Cantelli Lemma, that
$\P(\{deg_{\t_\infty}(n_0)=\infty\} \triangle
\{deg_{\t_\infty}(n_1)=\infty\}) = 0$.

To show that $\P(deg_{\t_\infty}(n_0)=\infty\} \triangle \{A_\infty
= \Z\})=0$ , we use a similar argument: Fix any $n_1\in \Z$. by
\eqref{eq:Ppmrecurr}
 \begin{eqnarray*}H_A(n_0) = \P_\infty(R(0) = n_0) = \sum_{x\notin
A}\P_\infty(R(0)=n_0,R(-1)=x) \\ = \sum_{x\notin
A}\frac{\P_\infty(T_x < T_A)}{\P_x(T_A < T_x)}p_{x,n_1}p_{n_1,n_0}
\geq  \sum_{x\notin A}\frac{\P_\infty(T_x < T_A}{\P_x(T_A <
T_x)}p_{x,n_1}p_{n_1,A} = \P_\infty(R(-1)=n_1) \end{eqnarray*}

We therefore conclude $\P(\{deg_{\t_\infty}(n_0)=\infty\} \triangle
\{\sum_{k\geq t_0} \P^{A_k}_\infty(R(-1)=n_1) = \infty\})=0$  and by
another use of Levy's extension $\P(\{deg_{\t_\infty}(n_0)=\infty\}
\triangle \{n_1\in A_\infty\})=0$
\end{proof}

\begin{rem*}
\begin{enumerate}
\item Since only finitely many points in $A_\infty$ can be close to
any point $n_0$, with some extra work the condition on the walk can
be weakened to $\liminf_{|n|\rightarrow
\infty}\frac{p_{n,0}}{p_{n,k}}
> c_k$ and $\limsup_{|n|\rightarrow \infty}\frac{p_{n,0}}{p_{n,k}} <
C_k$ for any $k\in \Z$.
\item A similar theorem holds for transient random walks ($\E(|\xi|) =
\infty$) using the gluing formula for transient walk (\cite{AABK}
section $7$)
\end{enumerate}
\end{rem*}

Using the above we get that for nice walks with finite variance, all
degrees in $\t_\infty$ are a.s. finite
\begin{coro}
If $\P(\xi>t) \lesssim t^{-\alpha}$ for $\alpha>3$, or $\P(\xi>t)
\approx t^{-\alpha}$ for $3> \alpha > 2$ then all degrees in
$\t_\infty$ are a.s. finite.
\end{coro}
\begin{proof}
For the $\alpha>3$ this follows from the renewal structure, since
after the first left and right renewal times,  only points within a
renewal interval can connect to points in that interval (with the
exception the renewal points themselves which have one connection to
the previous interval). For $2<\alpha<3$ \thmref{t:den_0} implies
$A_\infty \neq \Z$ a.s. and therefore by \propref{p:infinite} all
degrees in $\t_\infty$ are a.s. finite.
\end{proof}


\bibliographystyle{alpha}
\bibliography{DLA}

\noindent
{\bf Gideon Amir} \\
Department of Mathematics, Bar-Ilan University\\
Ramat Gan 52900, Israel \\
{\sc gidi.amir@gmail.com}\\

\end{document}